\documentclass[12pt,leqno]{article}
\usepackage{tikz}
\usepackage{hyperref}
\usepackage{empheq}
\usepackage{enumerate}
\usepackage{paralist}
\usepackage[doc]{optional}
\usepackage{xcolor}
\usepackage{float}
\usepackage{soul}
\usepackage{graphicx}
\definecolor{labelkey}{rgb}{0,0.08,0.45}
\definecolor{refkey}{rgb}{0,0.6,0.0}
\definecolor{Brown}{rgb}{0.45,0.0,0.05}
\definecolor{lime}{rgb}{0.00,0.8,0.0}
\definecolor{lblue}{rgb}{0.5,0.5,0.99}

 \usepackage{mathpazo}

\usepackage[T1]{fontenc}

\usepackage{amsmath}

\usepackage{stmaryrd}
\usepackage{amssymb}
\usepackage{theorem}
\newcommand{\menge}[2]{\big\{{#1}~\big |~{#2}\big\}}

\newcommand{\To}{\ensuremath{\rightrightarrows}}

\newcommand{\fenv}[1]%
{\ensuremath{\,\overrightarrow{\operatorname{env}}_{#1}}}
\newcommand{\benv}[1]%
{\ensuremath{\,\overleftarrow{\operatorname{env}}_{#1}}}

\newcommand{\scal}[2]{\left\langle{#1},{#2}  \right\rangle}

\newcommand{\RR}{\ensuremath{\mathbb R}}

\newcommand{\RPP}{\ensuremath{\mathbb{R}_{++}}}

\newcommand{\NN}{\ensuremath{\mathbb N}}

\newcommand{\dom}{\ensuremath{\operatorname{dom}}}

\newcommand{\prox}{\ensuremath{\operatorname{Prox}}}

\newcommand{\ran}{\ensuremath{\operatorname{ran}}}

\newcommand{\conv}{\ensuremath{\operatorname{conv}\,}}
\newcommand{\aff}{\ensuremath{\operatorname{aff}\,}}

\newcommand{\Id}{\ensuremath{\operatorname{Id}}}

\newcommand{\bR}{\ensuremath{{{R}}}}
\newcommand{\bT}{\ensuremath{{{T}}}}

\newcommand{\bJ}{\ensuremath{{{J}}}}

{\begin{list}{}{%
\settowidth{\labelwidth}{\textrm{#1~}}%
\setlength{\leftmargin}{\labelwidth+\labelsep}}}
{\end{list}}
\newtheorem{theorem}{Theorem}[section]

\newtheorem{lem}[theorem]{Lemma}

\newtheorem{cor}[theorem]{Corollary}
\newtheorem{proposition}[theorem]{Proposition}
\newtheorem{prop}[theorem]{Proposition}
\newtheorem{definition}[theorem]{Definition}
\newtheorem{defn}[theorem]{Definition}
\newtheorem{thm}[theorem]{Theorem}
\theoremstyle{plain}{\theorembodyfont{\rmfamily}
}
\theoremstyle{plain}{\theorembodyfont{\rmfamily}
}
\theoremstyle{plain}{\theorembodyfont{\rmfamily}
}
\theoremstyle{plain}{\theorembodyfont{\rmfamily}
\newtheorem{example}[theorem]{Example}}
\newtheorem{ex}[theorem]{Example}
\newtheorem{fact}[theorem]{Fact}
\theoremstyle{plain}{\theorembodyfont{\rmfamily}
}
\newtheorem{rem}[theorem]{Remark}

\providecommand{\ds}{\displaystyle}

\providecommand{\abs}[1]{\lvert#1\rvert}
\providecommand{\norm}[1]{\lVert#1\rVert}
\providecommand{\bk}[1]{\left(#1\right)}
\providecommand{\stb}[1]{\left\{#1\right\}}
\providecommand{\innp}[1]{\langle#1\rangle}

\providecommand{\RA}{\Rightarrow}

\providecommand{\RRR}{\left]-\infty,+\infty\right]}

\providecommand{\lam}{\lambda}
\providecommand{\RR}{\mathbb{R}}

\providecommand{\aff}{\operatorname{aff}}
\providecommand{\conv}{\operatorname{conv}}
\providecommand{\ran}{\operatorname{ran}}

\providecommand{\intr}{\operatorname{int}}

\providecommand{\dom}{\operatorname{dom}}

\providecommand{\gr}{\operatorname{gra}}
\providecommand{\gra}{\operatorname{gra}}
\providecommand{\Id}{\operatorname{{ Id}}}

\providecommand{\fady}{\varnothing}

\providecommand{\rras}{\rightrightarrows}
\providecommand{\To}{\rightrightarrows}

\providecommand{\DD}{{D}}

\providecommand{\NN}{\mathbb{N}}
\providecommand{\BB}[2]{\operatorname{ball}(#1;#2)}

\providecommand{\gr}{\operatorname{gra}}

\providecommand{\ran}{\operatorname{ran}}
\providecommand{\rec}{\operatorname{rec}}
\providecommand{\Id}{\operatorname{Id}}
\providecommand{\PR}{\operatorname{P}}
\providecommand{\pt}{{\partial}}

\providecommand{\inns}[2][w]{#2_{#1}}
\newcommand{\outs}[2][w]{{_#1}#2}

\providecommand{\T}{{ T}}
\providecommand{\D}{ {D}}
\providecommand{\E}{ {V}}
\providecommand{\U}{ {U}}
\providecommand{\uu}{ {u}}
\providecommand{\rra}{\rightrightarrows}

\providecommand{\fady}{\varnothing}

\providecommand{\ri}{\operatorname{ri}}

\providecommand{\RR}{\mathbb{R}}
\providecommand{\NN}{\mathbb{N}}

\definecolor{myblue}{rgb}{.8, .8, 1}
  \newcommand*\mybluebox[1]{%
    \colorbox{myblue}{\hspace{1em}#1\hspace{1em}}}

\allowdisplaybreaks 

\begin{document}

%
\title{\textsc
On the range of the Douglas--Rachford operator}

\author{
Heinz H.\ Bauschke\thanks{
Mathematics, University
of British Columbia,
Kelowna, B.C.\ V1V~1V7, Canada. E-mail:
\texttt{heinz.bauschke@ubc.ca}.},
Warren L.\ Hare\thanks{
Mathematics,
University of British Columbia,
Kelowna, B.C.\ V1V~1V7, Canada.
E-mail:  \texttt{warren.hare@ubc.ca}.},
~and Walaa M.\ Moursi\thanks{
Mathematics, University of
British Columbia,
Kelowna, B.C.\ V1V~1V7, Canada. E-mail:
\texttt{walaa.moursi@ubc.ca}.}}

\date{July 29, 2014}
\maketitle

\vskip 8mm

\begin{abstract} \noindent
The problem of finding a minimizer of the sum of two convex functions ---
or, more generally, that of finding a zero of the sum of two
maximally monotone operators --- is of central importance
in variational analysis. Perhaps the most popular method of
solving this problem is the Douglas--Rachford splitting method.
Surprisingly, little is known about the range of the
Douglas--Rachford operator.

In this paper, we set out to study this range systematically.
We prove that for $3*$ monotone operators a very pleasing formula
can be found that reveals the range to be nearly equal to a
simple set involving the domains and ranges of the underlying
operators. A similar formula holds for the range of the
corresponding displacement mapping. 
We discuss applications to subdifferential operators, to the infimal
displacement vector, and to firmly nonexpansive mappings. 
Various examples and counter-examples are presented, including
some concerning the celebrated Brezis--Haraux theorem. 
\end{abstract}

{\small

}
\section{Introduction}
Unless otherwise stated,
throughout this paper
\begin{empheq}[box=\mybluebox]{equation*}
X~~ \text{is a finite-dimensional real Hilbert space}
\end{empheq} 
with inner product $\innp{\cdot,\cdot}$ 
and associated norm $\norm{\cdot}$. 
Let $A:X\rras X$ be a set-valued operator. 
We say that $A$ is \emph{monotone} if 
$ \innp{x-y,u-v}\ge 0$ for all pairs $(x,u)$
and $(y,v)$ in $\gra A$, the \emph{graph} of $A$. 
A monotone operator 
$A$ is \emph{maximally monotone} if 
the graph of $A$, denoted $\gr A$,
 cannot be properly 
extended without destroying 
the monotonicity of $A$. 
Monotone operators are of considerable
importance 
in optimization and variational analysis;
see, e.g., 
\cite{BC2011},
\cite{Brezis},
\cite{BurIus}, 
\cite{Rock70},
\cite{Rock98},
\cite{Simons1},
\cite{Simons2},
\cite{Zalinescu},
\cite{Zeidler2a},
\cite{Zeidler2b},
and \cite{Zeidler1}. 
It is well known that the 
subdifferential operator
of a proper lower semicontinuous convex
function is maximally monotone. 
Subdifferential operators also belong to the class 
of $3^*$ monotone 
(also known as  rectangular) operators which
was introduced by Brezis and and Haraux \cite{Br-H}; see also 
\cite{JAT2012} and 
\cite{BYW12}.
The sum of two maximally monotone operators
is monotone; maximality, however, is guaranteed only
in presence of a constraint
qualification \cite{Rock:mm}.
The problem of finding the zeros of the sum 
of two maximally monotone operators is an
active topic in optimization as it captures 
the key problem of minimizing a sum of two convex functions. 
More broadly, 
from an optimization perspective, 
constrained optimization problems, convex
feasibility problem as well as
many other optimization 
problems can be interpreted and recast as
the problem of finding the zeros of the sum of two 
 maximally monotone operators.
Most methods for solving
the sum problem are splitting algorithms;
the most popular of which is the celebrated 
Douglas--Rachford method,
which was adopted to the monotone operator framework
by Lions and Mercier \cite{L-M97}. 
(See also e.g.\ \cite{BC2011}, 
\cite{Comb09}, 
\cite{Comb07},
\cite{Comb08}, 
and \cite{EckBer} 
for further results on and applications of this algorithm.)

Let $A\colon X\rras X$ be maximally monotone. 
Recall that
the \emph{resolvent} of 
$A$ is $\bJ_A = (\Id+A)^{-1}$,
where $\Id$ denotes 
the identity operator. 
Moreover, if $A$ is maximally monotone, 
then $\bJ_A$ is single-valued, firmly nonexpansive,
and maximally monotone. 
The \emph{reflected resolvent} of $A$ 
is $R_A = 2J_A-\Id$. 
Now let $B\colon X\rras X$ be also maximally monotone. 
The \emph{Douglas--Rachford splitting operator}
for the pair $(A,B)$ is 
\begin{equation}
T_{(A,B)}:=\tfrac{1}{2}\Id+\tfrac{1}{2}\bR_B\bR_A =
\Id-J_A+J_BR_A.
\end{equation} 
One main goal of this work is to analyze 
the ranges of $T_{(A,B)}$ and of 
the \emph{displacement mapping} $\Id- T_{(A,B)}$.
It is known 
(see, e.g., \cite[Corollary~2.14]{Sicon2014}
or \cite[Proposition~4.1]{EckThesis}) that 
\begin{equation}
\ran (\Id-\bT_{(A,B)})\subseteq 
(\dom A -\dom B)\cap (\ran A +\ran B). 
\end{equation}
It is natural to inquire whether or not this is a mere inclusion or
perhaps even an inequality. 
In general, this inclusion is strict --- sometimes even extremely so
in the sense that $\ran (\Id-\bT_{(A,B)})$ may be a singleton
while $(\dom A -\dom B)\cap (\ran A +\ran B)$ may be the entire
space; see Example~\ref{ex:1}. 
This likely has discouraged efforts to obtain a better
description of these ranges.
However, and somewhat surprisingly, we are able to obtain
--- under fairly mild assumptions on $A$ and $B$ --- 
the simple and elegant formulae 
\begin{empheq}[box=\mybluebox]{equation}
\ran (\Id-T_{(A,B)})\simeq (\dom A-\dom B)\cap(\ran A+\ran B)
\end{empheq}
and 
\begin{empheq}[box=\mybluebox]{equation}
\ran T_{(A,B)}\simeq (\dom A-\ran B)\cap(\ran A+\dom B),
\end{empheq}
where the ``near equality" of two sets $C$ and $D$, 
denoted by $C\simeq D$,
means that the two sets have the same 
relative interior
(the interior with respect to the closed affine hull)
 and the same closure.
When $A=\pt f$ and $B=\pt g$ are subdifferential operators, which
is the key setting in convex optimization, the above formulae
can be written as 
 \begin{empheq}[box=\mybluebox]{equation}
\ran (\Id-T_{(\pt f, \pt g)})\simeq (\dom f-\dom g)\cap(\dom f^*+\dom g^*)
\end{empheq}
and 
\begin{empheq}[box=\mybluebox]{equation}
\ran T_{(\pt f, \pt g)}\simeq (\dom f-\dom g^*)\cap(\dom f^*+\dom g).
\end{empheq}
These results are interesting because
the problem of finding a point in $(A+B)^{-1}(0)$ has
a solution if and only if 
$0\in \ran (\Id-T_{(A,B)})$
 (see, e.g., \cite[Lemma~2.6(ii)]{Comb04}). 
It also provides information on finding
the infimal displacement vector 
that defines the normal problem
recently introduced in \cite{Sicon2014}.
Moreover, $\ran T_{(A,B)}$ contains the set of fixed 
points of $T_{(A,B)}$.
Using the correspondence
between maximally monotone operators and
firmly nonexpansive mappings (see Fact~\ref{T:JA}),
we are able to reformulate our results 
for firmly nonexpansive mappings.
In addition to our main results,
we show that 
the well-known conclusion of Brezis-Haraux Theorem 
\cite{Br-H} is optimal in the sense that actual equality
may fail 
(see Example~\ref{Br-H-opt1} and Proposition~\ref{cn:ex5}).
Our investigation relies on the 
the class of $3^*$ monotone operators 
(see \cite{Br-H}),
Attouch--Th\'{e}ra duality (see \cite{AT}), and 
the associated 
normal problem (see \cite{Sicon2014}).

The remainder of this paper is organized as follows.
Section~\ref{S:M:T} contains a brief collection of
facts on monotone operators
and their resolvents, as well as  
on firmly nonexpansive mappings. 
In Section~\ref{S:eq:nc}, 
we review the notions of
 near convexity
and near equality, and we also present some new results. 
Section~\ref{S:AT:T} is concerned with the 
Attouch--Th\'{e}ra duality, the normal problem, 
 and the Douglas--Rachford splitting operator. 
Our main results are presented in  Section~\ref{S:Main},
while applications and special cases are provided in 
Section~\ref{S:app}. 
In Section~\ref{S:infd}, we offer some
results that are valid in a possibly infinite-dimensional
Hilbert space. 
We also provide various examples and counterexamples. 

We conclude this section with some comments on notation. 
We use $\PR_C$ and $N_C$
to denote projector and the normal cone operator 
associated with the nonempty closed convex subset $C$ of $X$.
The recession cone of $C$ is 
$\rec C :=\menge{x\in X}{x+C\subseteq C}$,
and the polar cone of $C$ is
$C^\ominus:=\menge{u\in X}{\sup\innp{C,u}\le 0}$.
We use $\BB{x}{r}$
to denote the closed
ball in $X$ centred at $x\in $ with radius $r>0$.
For a subset $S$ of $X$,
the relative interior of the set $S$ 
is 
$\ri S:=\menge{s\in S}{(\exists r>0)
\BB{s}{r}\cap\overline{\aff} S\subseteq S}$,
where $\aff S$ denotes the affine hull of $S$. 
All other notation is standard and follows, e.g., 
\cite{BC2011}. 

\section{Monotone operators and firmly nonexpansive mappings }

\label{S:M:T}

In this short section, we review some
key results on monotone operators and firmly nonexpansive
mappings that are needed subsequently.
(See also \cite{BC2011} for further results.)

\colorlet{sky}{white!85!blue}
\begin{fact}{\rm(See, e.g., \cite[Corollary~12.44]{Rock98}.)}
\label{maxm_rock}
Let $A:X\rra X$ and $B:X\rra X$ 
be maximally monotone such that 
$\ri \dom A\cap\ri \dom B\neq \fady$.
Then $A+B$ is maximally monotone.
\end{fact}

\begin{fact}{\rm (See, e.g., \cite[Proposition~23.2(i)]
{BC2011}.)}\label{F:res:ran}
Let $A:X\rras X$ be maximally monotone.
Then $\dom A=\ran \bJ_A$.
\end{fact}

Recall 
the inverse resolvent identity 
(see, e.g., \cite[Lemma~12.14]{Rock98})
\begin{equation}\label{inv:res}
\bJ_A+\bJ_{A^{-1}}=\Id.
\end{equation}
Applying Fact~\ref{F:res:ran}
to $A^{-1}$ and using \eqref{inv:res}, we obtain
\begin{equation} \label{inv:r:d}
\ran A=\dom A^{-1}=\ran \bJ_{A^{-1}}=\ran (\Id-\bJ_A).
\end{equation}

\begin{fact}[Minty parametrization]{\rm(See \cite{Minty}.)}
Let $A:X\rras X$ be maximally monotone. 
Then 
$\gra A \to X :(x,u)\to x+u$ 
is a continuous bijection, with continuous inverse
$x\mapsto(\bJ_A x,x-\bJ_A x)$, and 
\begin{equation}\label{Min:par}
\gra A=\menge{(\bJ_A x,x-\bJ_A x)}{x\in X}.
\end{equation}
\end{fact}

\begin{fact}{\rm(See \cite[Theorem~3.1]{Zar71:1}.)}
\label{rec:PC}
Let $C$ be a nonempty closed convex subset of $X$.
Then $\overline{\ran}(\Id-\PR_C)=(\rec C)^\ominus$.
\end{fact}

Let $T\colon X\to X$. 
Recall that $T$ is \emph{nonexpansive} if
$(\forall x\in X)(\forall y\in X)$
$\|Tx-Ty\|\leq \|x-y\|$,
and $T$ is \emph{firmly nonexpansive}
if 
\begin{equation}
(\forall x\in X)(\forall y\in X)\quad\|Tx-Ty\|^2 + \|(\Id-T)x-(\Id-T)y\|^2\leq \|x-y\|^2.
\end{equation}

\begin{fact}~{\rm(See, e.g., \cite[Theorem~2]{EckBer}.)}
\label{T:JA}
Let $\DD$ be a nonempty
subset of $ X$, let $\T:\DD\to  X$,
let $A\colon X\To X$, and suppose that
$A=\T^{-1}-\Id$. 
Then the following hold:
\begin{enumerate}
\item $\T=\bJ_{A}$. 
\item $A$ is monotone if and only if $\T$ is firmly nonexpansive.
\item $A$ is maximally monotone if and only if $\T$ is 
firmly nonexpansive and $\DD=X$. 
 \end{enumerate}
\end{fact}

\section{Near convexity and near equality}
\label{S:eq:nc}

We now review and extend results on
near equality and near convexity.

\begin{defn}[near convexity]
{\rm(See Rockafellar and Wets's \cite[Theorem~12.41]{Rock98}.)}
\label{def:n:con}
Let $D$ be a subset of $X$.
Then $D$ is \emph{nearly convex} if there exists a convex subset 
$C $ of $X$ such that $C\subseteq D\subseteq \overline{C}$.
\end{defn}

\begin{fact}
 {\rm(See \cite[Theorem~12.41]{Rock98}.)}
 \label{D-R-nc}
Let $A:X\rras X$ be maximally monotone.
Then $\dom A$ and $\ran A$ 
are nearly convex.
\end{fact}
\begin{defn}[near equality]
{\rm (See \cite[Definition~2.3]{BMW}.)}
\label{def:n:eq}
Let $C$ and $D$ be subsets of $X$. 
We say that $C$ and $D$ 
are \emph{nearly equal} if
\begin{equation}
C \simeq  D\;\; :\negthinspace\negthinspace\iff \;\; 
\overline{C}=\overline{D} \text{~~and~~}\ri C=\ri D.
\end{equation} 
\end{defn}
\begin{fact}{\rm(See \cite[Lemma~2.7]{BMW}.)}\label{lem:need}
Let $D$ be a nonempty nearly convex subset of $X$, say 
$C\subseteq D\subseteq \overline{C}$, where $C$
is a convex subset of $X$. Then
\begin{equation}
D\simeq  \overline{D} \simeq  \ri D \simeq  \conv D
\simeq  \ri \conv D \simeq  C.
\end{equation}
In particular, $\overline{D} $ and $\ri D$ are convex
and nonempty.
\end{fact}
\begin{fact}{ \rm(See \cite[Proposition~2.12(i)\&(ii)]{BMW}.)} \label{ri:part}
Let $C$ and $D$
be nearly convex subsets of X. Then 
\begin{equation}
C\simeq D \iff \overline{C}=\overline{D}.
\end{equation}
\end{fact}
\begin{fact}{\rm (See \cite[Lemma~2.13]{BMW}.)}\label{sum:nc}
Let $(C_i)_{i\in I}$ be a finite family of nearly 
convex subsets of $X$, and let 
$(\lam_i)_{i\in I}$ be a finite family of real numbers. 
Then $\sum_{i\in I} \lam_i C_i$
is nearly convex and $\ri(\sum_{i\in I}\lam_i C_i)
=\sum_{i\in I}\lam_i \ri C_i$.
\end{fact}
\begin{fact}
{\rm(See \cite[Theorem~2.14]{BMW}.)}
\label{prox:part} 
Let $(C_i)_{i\in I}$ be a finite family of nearly convex subsets 
of $X$, and let $(D_i)_{i\in I}$ be a family of subsets 
of $X$ such that $C_i\simeq D_i$ for every $i\in I$.
Then $\sum_{i\in I}C_i\simeq \sum_{i\in I}D_i$.
\end{fact}
\begin{fact}{\rm (See \cite[Theorem~6.5]{Rock70}.)}
\label{f:3}
Let $(C_i)_{i\in I}$ be a finite
 family of convex subsets of $X$. 
Suppose that $\cap_{i\in I}\ri C_i\neq \fady$. Then
$\overline{\cap_{i\in I} C_i}=\cap_{i\in I}\overline{ C_i}$
and 
$\ri \cap_{i\in I} C_i=\cap_{i\in I} \ri C_i$.
\end{fact}
Most of the following results are known. For different
proofs 
see also \cite[Thorem~2.1]{Bot2008} and 
the forthcoming \cite{Sarah} and \cite{MMW}. 

\begin{lem}\label{sh:pr}
Let $C$ and $\D$
be nearly convex subsets of $X$ such that
 $\ri C\cap \ri \D\neq \fady$. 
 Then the following hold:
 \begin{enumerate}
 \item \label{sh:pr:1}
 $C\cap D$ is nearly convex.
\item\label{sh:pr:4}
 ${C\cap \D}\simeq{ \ri C\cap \ri \D}$.
 \item\label{sh:pr:2}
 $\ri(C\cap D)=\ri C\cap \ri D$.
 \item\label{sh:pr:3}
 $\overline{C\cap D}=\overline{C}\cap \overline{D}$.
   \end{enumerate}
\end{lem}
\begin{proof}
\ref{sh:pr:1}:
Since $C$ and $\D$
are nearly convex, by Fact~\ref{lem:need},
$\ri C$ and $\ri \D$
are convex. Consequently, 
\begin{equation}\label{r:r:con}
\ri C\cap \ri \D \quad\text{is convex},
\end{equation}
 and clearly 
\begin{equation}\label{int:e:n}
\ri C\cap \ri \D\subseteq C\cap \D. 
\end{equation}
By Fact~\ref{lem:need} we have
$\ri C\simeq C$ and  $\ri D\simeq D$.
Hence, $\ri(\ri C)=\ri C$ 
and $\ri(\ri \D)=\ri \D$.
Therefore, 
\begin{equation}\label{int:e:1}
\ri(\ri C)\cap \ri(\ri \D)=\ri C\cap \ri \D\neq\fady.
\end{equation}
Using \eqref{int:e:1} and Fact~\ref{f:3} 
applied to the convex sets $\ri C$ and $\ri \D$
yield $\ri (C\cap D)=\ri C\cap \ri D$; hence
\begin{equation}\label{int:e:2}
\overline{\ri C\cap \ri \D}=\overline{\ri C}\cap\overline{ \ri \D}.
\end{equation}
Since $\ri C\simeq C$ and  $\ri D\simeq D$
by Fact~\ref{lem:need}, we have  
$\overline{\ri C}=\overline{C}$ 
and $\overline{\ri D}=\overline{D}$.  
Combining with \eqref{int:e:n}  and \eqref{int:e:2},
we obtain
\begin{equation}
\ri C\cap \ri \D\subseteq C\cap \D
\subseteq \overline{C}\cap \overline{\D}
= \overline{\ri C}\cap \overline{\ri \D}
=\overline{\ri C\cap \ri \D},
\end{equation}
which in turn yields \ref{sh:pr:1}
in view of \eqref{r:r:con}. 
\ref{sh:pr:4}:
Use \ref{sh:pr:1} and  Fact~\ref{lem:need} applied
to the convex set $\ri C\cap \ri D$
and the nearly convex set $C\cap D$.
\ref{sh:pr:2}:
Using \ref{sh:pr:4}, Fact~\ref{f:3}
applied to the convex sets 
$\ri C$ and $\ri D$
and \eqref{int:e:1} we have 
\begin{equation}
\ri(C\cap D)=\ri (\ri C\cap \ri D)=\ri (\ri C)\cap \ri (\ri D)=\ri C\cap \ri D,
\end{equation}
as required.
\ref{sh:pr:3}:
Since $C$ and $\D$ are nearly convex,
 it follows from Fact~\ref{lem:need}
that $ \overline{\ri C}= \overline{ C}$
and $\overline{ \ri \D}= \overline{ D}$.
Combining with \ref{sh:pr:4} and
\eqref{int:e:2} we have
\begin{equation}\label{intersect:e}
\overline{C\cap \D}=\overline{\ri C\cap \ri \D}
=\overline{\ri C}\cap\overline{ \ri \D}=
\overline{C}\cap \overline{\D},
\end{equation}
as claimed.
\end{proof}
\begin{cor}\label{eq:int}
Let $C_1$ and $C_2$ 
be nearly convex subsets of $X$,
 and let $\D_1$ and $\D_2$ be subsets of $X$
such that $C_1\simeq \D_1$ 
and $C_2\simeq \D_2$ .
Suppose that $\ri C_1\cap \ri C_2\neq \fady $. 
Then 
 \begin{equation}
 C_1\cap C_2\simeq \D_1\cap \D_2.
 \end{equation}
 \end{cor}
\begin{proof}
Let $i\in \stb{1,2}$. 
Since $C_i\simeq \D_i$, 
 by Definition \ref{def:n:eq}
 we have
\begin{equation}\label{d:int:e1}
\ri C_i=\ri \D_i\quad\text{and}\quad \overline{C_i}=\overline{D_i}.
\end{equation}
Hence,
\begin{equation}\label{int:int}
\ri \D_1\cap \ri D_2=\ri C_1\cap \ri C_2\neq \fady.
\end{equation}
Moreover, since  $C_i$ is nearly convex, 
it follows from Fact~\ref{lem:need}
that $\overline{\ri C_i}=\overline{ C_i}$
 and $\ri C_i$ is convex.
Therefore 
\begin{equation}
\ri C_i=\ri D_i\subseteq D_i\subseteq 
\overline{D_i}=\overline{C_i}=\overline{\ri C_i}.
\end{equation}
Hence $D_i$ is nearly convex. 
Applying 
Lemma~\ref{sh:pr}~\ref{sh:pr:2} 
to the two sets 
$C_1$ and $C_2$ implies that
\begin{equation}\label{d:int:e2}
\ri(C_1\cap C_2)=\ri C_1 \cap \ri C_2.
\end{equation}
Similarly we have
\begin{equation}\label{d:int:e3}
\ri(\D_1\cap \D_2)=\ri \D_1 \cap \ri \D_2.
\end{equation}
Using \eqref{int:int} and
Lemma~\ref{sh:pr}\ref{sh:pr:1}, 
applied to the sets $C_1$ and $C_2$
we have
$C_1\cap C_2$ 
is nearly convex.
Similarly, $ \D_1\cap \D_2$
is nearly convex.
By Fact~\ref{lem:need} 
$C_1\cap C_2\simeq \ri (C_1\cap C_2)$
and $D_1\cap D_2\simeq \ri (D_1\cap D_2)$.
Hence 
$\overline{C_1\cap C_2}= \overline{\ri (C_1\cap C_2)}$ and
$\overline{\D_1\cap \D_2}= \overline{\ri (\D_1\cap \D_2)}$.
Combining with \eqref{d:int:e2},
\eqref{int:int} and
\eqref{d:int:e3} yield
\begin{equation}
\overline{C_1\cap C_2}=\overline{\ri(C_1\cap C_2)}
=\overline{\ri C_1\cap \ri C_2}
=\overline{\ri \D_1\cap \ri \D_2}
=\overline{\ri(\D_1\cap \D_2)}
=\overline{\D_1\cap \D_2}.
\end{equation}
Now, Fact~\ref{ri:part} applied to the nearly convex
sets $C_1\cap C_2$ and $ \D_1\cap \D_2$
 implies that 
$C_1\cap C_2 \simeq  \D_1\cap \D_2$.
\end{proof}
\begin{defn}[{ $3^*$ monotone}]~{\rm(See \cite[page~166]{Br-H}.)}
Let $A:X\rra X$ be monotone. 
Then $A$ is \emph{$3^*$ monotone} 
(also known as \emph{rectangular}) if 
\begin{equation*}
(\forall x\in \dom A)(\forall v\in \ran A)
\qquad \inf_{(z,w)\in \gra A}\innp{x-z,v-w}>-\infty.
\end{equation*}
\end{defn}

\begin{fact}
 {\rm (See {\cite[page~167]{Br-H}}.)}
 \label{subd:star}
Let $f:X\to \RR$ be proper, convex, 
lower semicontinuous.
Then $\pt f$ is $3^*$ monotone.
\end{fact}
\begin{fact}\label{Br-H}{\rm\bf ({Brezis--Haraux})} 
Let $H$ be a real (not necessarily 
finite-dimensional) Hilbert space,
let $A:H\rra H$ and $B:H\rra H$ 
be monotone operators 
such that $A+B$ is 
maximally monotone and one 
of the following conditions holds:
\begin{enumerate}
\item $A$ and $B$ are $3^*$ monotone.
\item $\dom A \subseteq \dom B$ 
and $B$ is $3^*$ monotone.
\end{enumerate}
Then 
\begin{equation}\label{Br-H-H}
\overline{\ran}(A+B)=\overline{\ran A+\ran B}  
\quad \text{and}\quad  \intr\ran(A+B )=\intr (\ran A+\ran B).
\end{equation}
If $H$ is finite-dimensional, then 
\begin{equation}\label{Br-H-X}
\ran (A+B)~\text{is nearly convex 
\quad and}\quad  \ran(A+B)\simeq \ran A+\ran B.
\end{equation}
\end{fact}
 \begin{proof}
See {\rm\cite[Theorems~3 and 4]{Br-H}} for the proof
of \eqref{Br-H-H}
and {\rm\cite[Theorem~3.13]{BMW}}
for the proof
of \eqref{Br-H-X}.
 \end{proof}

Example~\ref{Br-H-opt1} and 
Proposition~\ref{cn:ex3}
illustrate that 
the results of Fact~\ref{Br-H}
are optimal in the sense that actual equality fails.

 \begin{ex}\label{Br-H-opt1}
 Suppose that $X=\RR^2$ and let 
 $f:\RR^2\to]-\infty,+\infty]:(\xi_1,\xi_2)
 \mapsto\max\stb{g(\xi_1),\abs{\xi_2}}$,
 where $g(\xi_1)=1-\sqrt{\xi_1}$ if $\xi_1\ge 0$, 
 $g(\xi_1)=+\infty$ if $\xi_1< 0$. 
 Set $A=\pt f^*$. Then
 $A$ is $3^*$ monotone and
 $2A=A+A$ is maximally monotone,
yet
 \begin{equation}
 2\ran A=\ran 2A=\ran (A+A)\subsetneqq \ran A+\ran A.
 \end{equation}
 \end{ex}
 \begin{proof}
 First notice that
 by Fact~\ref{subd:star} 
$A$ is $3^*$ monotone.
Moreover, since $A$ is maximally monotone, it
follows from
Fact~\ref{lem:need} that
$\ri\dom A $ is nonempty and convex.
Since 
$\ri \dom A\cap \ri \dom A=\ri \dom A\neq \fady$,
Fact~\ref{maxm_rock} implies that $A+A=2A$ 
 is maximally monotone.
 Using \cite[Proposition~16.24]{BC2011} and
  \cite[example on page~218]{Rock70} we know that 
 $\ran \pt f^*=\ran (\pt f)^{-1}=
 \dom \pt f= \menge{(\xi_1,\xi_2)}{
 \xi_1> 0}\cup
 \menge{(0,\xi_2)}{\abs{\xi_2}\ge 1}$
 and 
 $\ran (A+A)= \menge{(\xi_1,\xi_2)}{
 \xi_1> 0}\cup
 \menge{(0,\xi_2)}{
\abs{\xi_2}\ge 2}\neq \menge{(\xi_1,\xi_2)}{
 \xi_1\ge 0}=\ran A+\ran A$.
\end{proof}
 \section{Attouch--Th\'{e}ra duality and the normal problem}
 \label{S:AT:T}

This section provides a review of the Attouch--Th\'era duality
and the associated normal problem.
From now on, we assume that
\begin{empheq}[box=\mybluebox]{equation*}
A:X\rras X \text{~~and~~} B:X\rras X
\text{~~~are maximally monotone}.
\end{empheq}
We abbreviate
\begin{equation*}
A^\ovee:=(-\Id)\circ A\circ (-\Id), \quad
A^{-\ovee}:=(A^{-1})^\ovee=(A^\ovee)^{-1}, 
\end{equation*}
and we observe that $A^\ovee$ is maximally monotone
as is $(A^{-1})^\ovee=(A^\ovee)^{-1}$.
The \emph{primal problem} associated with
the ordered pair $(A,B)$
is to find the zeros of $A+B$. 
Since $A$ and $B$ are maximally monotone operators,
so are $A^{-1}$ and $B^{-\ovee}$.
The \emph{dual pair}
of $(A,B)$ is defined by 
\begin{equation*}
\label{dual:op}
(A,B)^* := (A^{-1},B^{-\ovee}).
\end{equation*}
We now recall the definition of the dual problem.
\begin{definition}[(Attouch--Th\'era) dual problem]
The \emph{(Attouch--Th\'era) dual problem} 
for the ordered pair $(A,B)$
is to find the set of zeros of $A^{-1}+B^{-\ovee}$.
\end{definition}
From now on, we shall use $T_{(A,B)}$
to refer to \emph{the Douglas--Rachford splitting 
operator} for two operators $
A$ and $B$, defined as
\begin{empheq}[box=\mybluebox]{equation}
\label{e:done}
T_{(A,B)}:=\tfrac{1}{2}\Id+\tfrac{1}{2}\bR_B\bR_A
= \Id-J_A + J_BR_A. 
\end{empheq}

\begin{fact}[self-duality]
{\rm(See \cite[Lemma~3.6~on~page~133]{EckThesis}
or \cite[Corollary~4.3]{JAT2012}.)}
\label{s:dual}
\begin{equation}\label{eq:dual}
T_{(A^{-1},B^{-\ovee})}
= T_{(A,B)}. 
\end{equation}
\end{fact}
\begin{fact}{\rm (See \cite[Proposition~4.1]{EckThesis} or
\cite[Corollary~2.14]{Sicon2014}.)}
\begin{align}
\ran (\Id-T_{(A,B)})&=\menge{a-b}{(a,a^*)\in \gr A,
 (b,b^*)\in \gr B ,~a-b=a^* +b^*}\label{con:eq}\\
&\subseteq (\dom A -\dom B)\cap (\ran A +\ran B).\label{easyinc}
\end{align}
\end{fact}

\begin{fact}{\rm(See \cite[Proposition~2.16]{Sicon2014}.)}
\label{T:vs:Id:T}
$T_{(A,B)}=\Id-T_{(A,B^{-1})}$.
\end{fact} 

Let $w\in X$ be fixed. For the operator 
$A$, the \emph{inner and outer 
shifts} associated with $A$ are defined by 
\begin{align}
\inns[w]{A}\colon& X\To X \colon x\mapsto A(x-w),\\
\outs[w]{A}\colon & X\To X \colon x\mapsto Ax-w.
\end{align}
Notice that $\inns[w]{A}$ and 
$\outs[w]{A}$ are maximally monotone, with
$\dom \inns[w]{A}=\dom A+w$ and $\dom \outs[w]{A}=\dom A$. 

\begin{defn}[The $w$-perturbed problem]
{\rm{(See \cite[Definition~3.1]{Sicon2014}.)}}
The \emph{$w$-perturbed problem}
associated with the pair $(A,B)$
is to determine the set of zeros 
\begin{equation}\label{defn:Zw}
Z_w:=( \outs[w]{A} +\inns[w]{B}  )^{-1}(0)
=\menge{x\in X}{ w\in A x +B(x-w)}.
\end{equation}
\end{defn}
\begin{fact}{\rm (See \cite[Proposition~3.3]{Sicon2014}.)}
\label{chrac:1}
Let $w\in X$. Then
\begin{equation}\label{def:zw}
Z_w\neq \fady \iff w\in \ran (\Id-T_{(A,B)})\iff w\in \ran(A+\inns{B}).
\end{equation}
\end{fact}
\begin{defn}[normal problem and infimal displacement vector]
\label{def:n:p}
The \emph{normal problem} associated 
with the pair $(A,B)$
is the $v_{(A,B)}$-perturbed problem, where
$v_{(A,B)}$ is \emph{the infimal displacement vector}
for the pair $(A,B)$ defined as
\begin{equation}\label{defn:v}
v_{(A,B)}:=\PR_{\overline{\ran}(\Id-T_{(A,B)})}(0).
\end{equation}
\end{defn}

\begin{fact}{\rm(See \cite[Proposition~3.11]{Sicon2014}.)}
\label{eq:norm}
\begin{equation}
\norm{v_{(A,B)}}=\norm{v_{(B,A)}}.
\end{equation}
\end{fact}
In view of Definition~\ref{def:n:p} and Fact~\ref{eq:norm}, 
the magnitude of the vector $v_{(A,B)}$
is actually a measure of how far 
the original problem is from
the normal problem. This magnitude is the same 
for the pairs $(A,B)$ and $(B,A)$.

We now explore how $\ran(\Id-\T)$ 
is related to the set $(\dom A-\dom B)\cap (\ran A+\ran B)$. 
We will prove that when the operators $A$ and $B$
 are ``sufficiently nice'', we have
\begin{equation}\label{awl:1}
 \ran({\Id -T_{(A,B)}})\simeq(\dom A -\dom B)\cap (\ran A +\ran B).
 \end{equation}

In general, \eqref{awl:1} may fail spectacularly as we will now illustrate. 
\begin{ex}\label{ex:1}
Suppose that $X=\RR^2$,
and that 
\begin{equation}
A=\begin{pmatrix}
0&-1\\
1&0
\end{pmatrix}
\quad\text{and}\quad
B=-A=\begin{pmatrix}
0&1\\
-1&0
\end{pmatrix}.
\end{equation}
Then 
\begin{equation}
\ran(\Id-T_{(A,B)})=\stb{0}\subsetneqq \RR^2=
(\dom A -\dom B)\cap (\ran A +\ran B).
\end{equation}
\end{ex}
\begin{proof}
Recall that $\dom A=\dom B =\ran A =\ran B=\RR^2$, 
consequently $(\dom A -\dom B)\cap 
(\ran A +\ran B)=\RR^2$. 
On the other hand, 
one checks that 
$\bR_A:(x,y)\mapsto(y,-x)=B$ 
 and $\bR_B:(x,y)\mapsto(-y,x)=A$.
Hence $\bR_B\bR_A=\Id$ and therefore  
 $\Id-T_{(A,B)}=\tfrac{1}{2}(\Id-\bR_B\bR_A)\equiv0$. 
\end{proof}

\section{Main results}\label{S:Main}

Upholding the notation of Section~\ref{S:AT:T}, we also 
set 
\begin{empheq}[box=\mybluebox]{equation}\label{def:D:R}
D:=D_{(A,B)}:=\dom A-\dom B\quad\text{and}\quad R:=R_{(A,B)}:=\ran A+\ran B.
\end{empheq}

We start by proving some auxiliary 
results.
\begin{lem}\label{lem:HB}
The following hold:
\begin{enumerate}
 \item\label{lem:HB:2}
 The sets $D$ and $R$
are nearly convex.
 \item\label{lem:HB:1}
 $\ri D\cap \ri R\neq \fady$.
 \item\label{lem:HB:3}
 $D\cap R$ is nearly convex.
 \item\label{lem:HB:6}
 $\ri(D\cap R)=\ri D\cap \ri R$ .
 \item\label{lem:HB:4}
$\overline{D\cap R}=\overline{\ri D\cap\ri R}$.
\item\label{lem:HB:5}
$\overline{D\cap R}=\overline{D}\cap \overline{R}$.
\end{enumerate} 
\end{lem}
\begin{proof}
\ref{lem:HB:2}: Combine Fact~\ref{D-R-nc} 
and Fact~\ref{sum:nc}.
\ref{lem:HB:1}: Since $B$ is maximally monotone,
the Minty parametrization 
\eqref{Min:par} implies that $X=\dom B+\ran B$.
Hence by \ref{lem:HB:2}
and Fact~\ref{sum:nc}
\begin{equation}\label{ri:1}
0\in X=\ri X=\ri(\ran A+\ran B-(\dom A-\dom B))=\ri R-\ri D.
\end{equation}
Hence, $\ri D\cap \ri R\neq \fady $, as claimed.
(Note that we did not use the maximal
monotonicity of $A$ in this proof.)
\ref{lem:HB:3}: Combine \ref{lem:HB:2}, \ref{lem:HB:1}
and Lemma~\ref{sh:pr}\ref{sh:pr:1}.
\ref{lem:HB:6}: Combine \ref{lem:HB:2}, \ref{lem:HB:1}
and Lemma~\ref{sh:pr}\ref{sh:pr:2}.
\ref{lem:HB:4}: Combine \ref{lem:HB:2}, \ref{lem:HB:1}
and Lemma~\ref{sh:pr}\ref{sh:pr:4}.
\ref{lem:HB:4}: Combine \ref{lem:HB:2}, \ref{lem:HB:1}
and Lemma~\ref{sh:pr}\ref{sh:pr:3}.
\end{proof}

\begin{thm}\label{main:conc}
Suppose that $A$ and $B$ satisfy
one of the following:
\begin{enumerate}
\item\label{abs:line}
$(\forall w\in \ri D\cap \ri R)\quad\ri (\ran A+\ran B)
\subseteq\ri\ran\bk{A+\inns{B}}.$
\item\label{main:c:1}
$A $ and $B$ are $3^*$ monotone.
\item\label{main:c:2}
$\dom B+\ri D\cap \ri R
\subseteq \dom A$ and $A$ is $3^*$ monotone.
\end{enumerate}
Then 
\begin{equation}\label{eq:nq}
\ran(\Id-T_{(A,B)})\simeq D\cap R. 
\end{equation}
Furthermore, the following implications hold:
\begin{equation}\label{e:add:t}
(\exists C\in \stb{A,B})~\dom C=X \text{~~and~~} C\text{~~is~~} 3^*
\text{~~monotone~~} \;\;\Longrightarrow\;\; \ran(\Id-T_{(A,B)})\simeq R,
\end{equation}
\begin{equation}\label{e:add:t:d}
(\exists C\in \stb{A,B})~\ran C=X \text{~~and~~} C\text{~~is~~} 3^*
\text{~~monotone~~} \;\;\Longrightarrow\;\; \ran(\Id-T_{(A,B)})\simeq D,
\end{equation}
and 
\begin{equation}\label{e:equality}
\ri(D\cap R)=D\cap R
\;\;\Longrightarrow\;\;
\ran(\Id-T_{(A,B)})= D\cap R. 
\end{equation}
\end{thm}
\begin{proof}
First we show that
 \begin{equation}\label{maxmono}
 (\forall w\in \ri D)\quad
 A+\inns{B} \quad \text{is maximally monotone}.
 \end{equation}
Notice that $(\forall w\in X)\dom\inns{B}
=\dom B+w$. 
Let $w\in \ri D=\ri (\dom A-\dom B)$.
Then $\ri \dom A\cap \ri \dom \inns{B}\neq \fady$. 
Using Fact~\ref{maxm_rock}, we conclude 
that  $A+\inns{B}$ is maximally monotone,
which proves \eqref{maxmono}.
Now, suppose that \ref{abs:line} holds.
Then $(\forall w\in D\cap R)$  
$w\in \ri \ran (A+\inns{B})\subseteq \ran (A+\inns{B})$.
Combining with Fact~\ref{chrac:1} we
conclude that $(\forall w\in \ri D\cap \ri R)$
$w\in \ran (\Id-T_{(A,B)})$. Hence 
 \begin{equation}
\overline{\ri D\cap \ri R}\subseteq \overline{\ran} (\Id-T_{(A,B)}).
\end{equation}
It follows from Lemma~\ref{lem:HB}\ref{lem:HB:4}
that
$\overline{\ri D\cap \ri R}=\overline{D\cap R}$.
Altogether,  
\begin{equation}
\overline{D\cap R}\subseteq \overline{\ran}(\Id-T_{(A,B)}).  
\end{equation}
It follows from \eqref{easyinc} that
$ \overline{\ran}(\Id-T_{(A,B)})\subseteq \overline{D\cap R}$.
Therefore,
\begin{equation}\label{cl:part}
\overline{D\cap R}= \overline{\ran}(\Id-T_{(A,B)}).  
\end{equation}
Since $T_{(A,B)}$ is firmly nonexpansive, 
hence nonexpansive,
it follows from {\rm\cite[Example~20.26]{BC2011}}
that $\Id-T_{(A,B)}$ is maximally monotone, 
and therefore $\ran(\Id-T_{(A,B)})$ 
is nearly convex by Fact~\ref{D-R-nc}. 
Using Lemma~\ref{lem:HB}\ref{lem:HB:1}\&\ref{lem:HB:3}  
we know that $\ri D\cap \ri R\neq \fady$
and
 $D\cap R$ is nearly convex.
Therefore, using \eqref{cl:part} and Fact~\ref{ri:part}  
applied to the nearly convex sets $D\cap R$
 and $\ran(\Id-T_{(A,B)})$, 
we get  $\ran(\Id-T_{(A,B)})\simeq D\cap R$.  
Now we show 
that each of the conditions
\ref{main:c:1} and \ref{main:c:2}
imply \ref{abs:line}.
Let $w\in \ri D\cap \ri R$, and 
notice that \ref{main:c:1}
implies that
 $\inns{B}$ is $3^*$ monotone,
 whereas \ref{main:c:2} 
 implies that
 $\dom \inns{B}=\dom B+w\subseteq \dom A$.
 Using \eqref{maxmono} and 
 Fact~\ref{Br-H} applied to
 $A$ and $\inns{B}$
we have $(\forall w\in \ri D\cap \ri R)$
\begin{equation}
 w\in \ri R
 = \ri (\ran A+\ran B)
      =\ri (\ran A +\ran \inns{B})
  =\ri \ran (A+\inns{B}).
 \end{equation}
 That is, \ref{abs:line} holds, and consequently  
 \eqref{eq:nq} holds.
 
 We now turn to the implication \eqref{e:add:t}. 
 Observe first that $D=X$. 
If $A$ is $3^*$ monotone and $\dom A=X$, then
clearly \ref{main:c:2} holds.
Thus, it remains to consider the case when
$B$ is $3^*$ monotone and $\dom B=X$.
Then $\inns{B}$ is $3^*$ monotone and
$\dom A\subseteq X = \dom \inns{B}$.
As before, we obtain 
 $w\in \ri R = \ri (\ran A+\ran B)
      =\ri (\ran A +\ran \inns{B})
  =\ri \ran (A+\inns{B})$.
Hence \ref{abs:line} holds, which completes
the proof of \eqref{e:add:t}. 
To prove the implication \eqref{e:add:t:d}, 
first notice that 
$(\exists C\in \stb{A,B})~\ran C=X$ and $C$ is $3^*
$ monotone 
$\iff(\exists C\in \stb{A^{-1},B^{-\ovee}})~\dom C=X$ and $C$ is $ 3^*$ monotone. 
Therefore using Fact~\ref{s:dual} and \eqref{e:add:t}
applied to the operators $A^{-1}$ and $B^{-\ovee}$
$(\exists C\in \stb{A,B})~\ran C=X \text{~~and~~} C\text{~~is~~} 3^*
\text{~~monotone~~} \RA \ran (\Id-T_{(A,B)}) =
\ran (\Id-T_{(A^{-1},B^{-\ovee}})=R_{(A^{-1},B^{-\ovee})}=
\ran A^{-1}+\ran B^{-\ovee}=\dom A-\dom B=D$, which proves \eqref{e:add:t:d}.
 
 Now suppose that $\ri(D\cap R)=D\cap R$.
 It follows from \eqref{eq:nq}
and \eqref{easyinc}
that
\begin{equation}
D\cap R=\ri(D\cap R)
=\ri\ran (\Id-T_{(A,B)})\subseteq 
\ran (\Id-T_{(A,B)})\subseteq D\cap R.
\end{equation}
Hence all the inclusions become equalities, 
which proves \eqref{e:equality}. 
\end{proof}

\begin{cor}[range of the Douglas--Rachford operator]\label{ran:T}
Suppose that $A$ and $B$ satisfy
one of the following:
\begin{enumerate}
\item\label{abs:line:d}
$(\forall w\in \ri D_{(A,B^{-1})}\cap \ri R_{(A,B^{-1})})\quad\ri (\ran A+\dom B)
\subseteq\ri\ran\bk{A+\inns{B^{-1}}}.$
\item\label{main:d:1}
$A $ and $B$ are $3^*$ monotone.
\item\label{main:d:2}
$\ran B+\ri D_{(A,B^{-1})}\cap \ri R_{(A,B^{-1})}
\subseteq \dom A$ and $A$ is $3^*$ monotone.
\end{enumerate}
Then
\begin{equation}
\ran T_{(A,B)}\simeq (\dom A-\ran B)\cap(\ran A+\dom B).
\end{equation}
Furthermore, the following implications hold:
\begin{equation}
(\exists C\in \{A,B^{-1}\})\quad  \dom C=X\,\,\text{and $C$ is $3^*$
monotone}
\;\;\Longrightarrow\;\;
\ran T_{(A,B)}\simeq\ran A+\dom B
\end{equation}
and 
\begin{multline}
\label{e:T:equality}
\ri\big(D_{(A,B^{-1})}\cap R_{(A,B^{-1})}\big) = 
D_{(A,B^{-1})}\cap R_{(A,B^{-1})}\\
\;\;\Longrightarrow\;\;
\ran T_{(A,B)}=(\dom A-\ran B)\cap(\ran A+\dom B).
\end{multline}
\end{cor}
\begin{proof}
Using Fact~\ref{T:vs:Id:T}, 
we know that $T_{(A,B)}=\Id-T_{(A,B^{-1})}$.
The result thus follows by applying
Theorem~\ref{main:conc} to $(A,B^{-1})$.
\end{proof}

The assumptions in Theorem~\ref{main:conc}
are critical. Example~\ref{ex:1} shows that 
when neither $A$ nor $B$ is $3^*$ monotone,
the conclusion of the theorem fails.
Now we show that the conclusion 
of Theorem~\ref{main:conc}
fails even if one of the operators is a subdifferential
operator. 

\begin{ex}\label{ex:2}
Suppose that $X=\RR^2$,
set $C=\RR\times \stb{0}$, and
suppose that 
\begin{equation}
A=\begin{pmatrix}
0&-1\\
1&0
\end{pmatrix}
\quad\text{and}\quad
B=N_C.
\end{equation}
Then $\Id-T_{(A,B)}=J_A-\PR_CR_A$.
Notice that
$\PR_C:(x,y)\mapsto(x,0)$,
$J_A:(x,y)\mapsto \tfrac{1}{2}(x+y,-x+y)$ and 
consequently $R_A:(x,y)\mapsto(y,-x)$. 
Hence
\begin{equation}
\ran(\Id-T_{(A,B)})=\RR\cdot(1,-1)\subsetneqq \RR^2=
(\dom A -\dom B)\cap (\ran B +\ran A).
\end{equation}
\end{ex}

\begin{cor}\label{aff:case}
Suppose that $A$ and $B$ satisfy
one of the following:
\begin{enumerate}
\item
$(\forall w\in \ri D\cap \ri R)\quad\ri (\ran A+\ran B)
\subseteq\ri\ran\bk{A+\inns{B}}.$
\item
$A $ and $B$ are $3^*$ monotone.
\item
$\dom B+\ri D\cap \ri R
\subseteq \dom A$ and $A$ is $3^*$ monotone.
\item
$(\exists C\in \stb{A,B})$ $ \dom C=X$ and $C$ is $3^*$ monotone.
\end{enumerate}
Furthermore, suppose that $D$ and $R$ are affine subspaces.
Then $\ran(\Id-T_{(A,B)})= D\cap R$. 
\end{cor}
\begin{proof}
Since $\ri D=D$ and $\ri R=R$,
Lemma~\ref{lem:HB}\ref{lem:HB:6} yields
$D\cap R=\ri D\cap \ri R=\ri(D\cap R)$.
Now apply
\eqref{e:equality}.
\end{proof}

\begin{cor}
Suppose that
$X=\RR$. Then $\ran (\Id-T_{(A,B)})\simeq D\cap R$.
\end{cor}
\begin{proof}
Indeed, it follows from e.g.\ 
\cite[Corollary~22.19]{BC2011}
 and Fact~\ref{subd:star} that $A$
 and $B$ are $3^*$ monotone. Now 
 apply Theorem~\ref{main:conc}\ref{main:c:1}.
\end{proof}

We now construct an example where 
$\ran(\Id-T_{(A,B)}) $ 
properly lies between
$\ri(D\cap R)$ and $\overline{D\cap R}$.
This illustrate that 
Theorem~\ref{main:conc} 
is optimal in the sense that
near equality cannot be replaced by actual equality.

\begin{ex}\label{cn:ex1}
Suppose that $\dim X\geq 2$,
let $u$ and $v$ be in $X$ with $u\neq v$, 
let $r$ and $s$ be in $\RPP$,
set $U=\BB{u}{r}$ and
$V=\BB{v}{s}$, and suppose that
$A=N_U$ and $B=N_V$. 
Then $D\cap R = \BB{u-v}{r+s}$ and 
\begin{equation}
\ran(\Id-T_{(A,B)})=\intr \BB{u-v}{r+s}
\cup\ds\stb{\bk{1-
\tfrac{r+s}{\norm{u-v}}}(u-v),\bk{1+\tfrac{r+s}{\norm{u-v}}}(u-v)};
\end{equation}
consequently,
\begin{equation}
\ri (D\cap R)\subsetneqq  \ran(\Id-T_{(A,B)})
\subsetneqq \overline{D\cap R}.
\end{equation}
Moreover,
\begin{equation}
v_{(A,B)}=\max\big\{ (r+s) - \norm{v-u} , 0\big\}\cdot
\frac{v-u}{ \norm{v-u}}. 
\end{equation}
\end{ex}
\begin{proof}
It follows from Fact \ref{subd:star}
that $A$ and $B$ are $3^*$ monotone.
Using e.g.\ \cite[Corollary~21.21]{BC2011}, 
we have 
$\ran A=\ran B=X$, hence $R=X$ and $D\cap R=D=U-V$.
First notice that
\begin{equation}\label{U-V}
D\cap R=D=U-V=\BB{u-v}{r+s}.
\end{equation}
We claim that
\begin{equation}\label{sing}
(\forall w\in D\setminus \ri D)\quad 
U\cap(V+w)\quad \text{is a singleton}.
\end{equation}
Since $D=U-V$, we have  
$(\forall w\in D) $ $
U\cap(V+w)\neq \fady$.
Now let $w\in D\setminus \ri D$
and assume to the contrary that 
$\stb{y,z}\subseteq U\cap(V+w)$
with $y\neq z$. 
Then 
$\{y-w,z-w\} \subseteq V$,
and
$(\forall \lam \in \left ]0,1\right [)$
\begin{equation}
\lam y+(1-\lam )z\in\intr U\quad \text {and }
\quad \lam y+(1-\lam )z-w\in \intr V.
\end{equation}
It follows from Fact~\ref{sum:nc}
and the above inclusions that 
$w\in  \intr U-\intr V=\ri U-\ri V=\ri D$, 
which is absurd.
Therefore \eqref{sing} holds.
Now, let $w\in D\setminus \ri D$
and notice that $V+w=\BB{v+w}{r}$.
Using \eqref{sing} we have 
$U\cap(V+w)=\dom(A+\inns{B})$ 
is a singleton. 
Consequently, $\ran(A+\inns{B})$ 
is the line passing through the origin
parallel to the line 
passing through 
$u$ and $v+w$, and 
by Fact~\ref{chrac:1}, we have
$w\in \ran (\Id-T_{(A,B)})$
$\iff w\in\ran(A+\inns{B})$ 
$\iff w=\lam (u-v-w)$ for some 
$\lam \in \RR\setminus\stb{-1}$
$\iff w=\tfrac{\lam }{1+\lam}(u-v)$ 
with $\lam \in \RR\setminus\stb{-1}$
(since $u\neq v$),
or equivalently,
\begin{equation}\label{w:u:v}
w=\alpha(u-v),\quad \text{where~} \alpha\in \RR\setminus\stb{1}.
\end{equation}
Finally notice that $w$ is on the boundary of 
$U-V$. Therefore, using \eqref{U-V} 
and \eqref{w:u:v} we must have
$\norm{w-(u-v)}=r+s$ 
$\iff \abs{\alpha-1}\norm{u-v}=r+s$
$\iff \alpha=1\pm \frac{r+s}{\norm{u-v}}$,
which means that only two points on the boundary 
of $D$ are included in $\ran (\Id-T_{(A,B)})$.
Moreover, if $\norm{u-v}> r+s$,
then $0\in \intr \BB{u-v}{s+r}$,
hence $v_{(A,B)}=0$.
Else if $\norm{u-v}\le r+s$,
using \cite[Proposition~28.10]{BC2011}
we get
$
v_{(A,B)}=
\bk{1- \tfrac{r+s}{\norm{u-v}}}(u-v)
$,
which completes the proof. 
\end{proof}

\section{Applications}

\label{S:app}
\subsection{On the infimal displacement vector $v_{(A,B)}$}

In this section, we focus on $v_{(A,B)}$.

\begin{proposition}
\label{p:kaos}
Suppose that $A$ and $B$ satisfy
one of the following:
\begin{enumerate}
\item 
$(\forall w\in \ri D\cap \ri R)\quad\ri (\ran A+\ran B)
\subseteq\ri\ran\bk{A+\inns{B}}.$
\item
$A $ and $B$ are $3^*$ monotone.
\item 
$\dom B+\ri D\cap \ri R
\subseteq \dom A$ and $A$ is $3^*$ monotone.
\item 
$(\exists C\in \stb{A,B})$ $ \dom C=X$ and $C$ is $3^*$ monotone.
\end{enumerate}
Then 
$\overline{\ran}(\Id-T_{(A,B)})
=\overline{D\cap R}=\overline{D}\cap \overline{R}$ and 
$v_{(A,B)}=\PR _{\overline{D}\cap\overline{R}}(0)$.
\end{proposition}
\begin{proof}
Combine Theorem~\ref{main:conc},
Lemma~\ref{lem:HB}\ref{lem:HB:5}, and 
\eqref{defn:v}.
\end{proof}

Using the symmetric hypotheses of Theorem~\ref{main:conc}, 
we obtain the following result:

\begin{lem}\label{new:f:v}
Suppose that both $A$ and $B$ are $3^*$ monotone, 
or that $(\exists C\in \{A,B\})$ such that $ \dom C=X$ 
and $C$ is $3^*$ monotone.
Then the following hold:
\begin{enumerate}
\item\label{new:f:v:2}
 If $D$ is a linear subspace of $X$, then  
 $\ran(\Id-T_{(A,B)})\simeq \ran(\Id-T_{(B,A)})$
 and 
$v_{(A,B)}=v_{(B,A)}$.
  \item\label{new:f:v:3}
   If $R$ is a linear subspace of $X$, 
 then
$\ran(\Id-T_{(A,B)})\simeq -\ran(\Id-T_{(B,A)})$
and $v_{(A,B)}= - v_{(B,A)}$. 
\item\label{new:f:v:4}
If $\dom A=X$ or $\dom B=X$, 
then $\ran(\Id-T_{(A,B)})\simeq \ran(\Id-T_{(B,A)})\simeq R$,
and $v_{(A,B)}=v_{(B,A)}=\PR_{\overline{R}}(0)$.
\item\label{new:f:v:5}
If $\dom A$ or $\dom B$ is bounded, then
$\ran(\Id-T_{(A,B)})\simeq -\ran(\Id-T_{(B,A)})\simeq D$,
and 
 $v_{(A,B)}= - v_{(B,A)}=\PR_{\overline{D}}(0)$. 
\end{enumerate}
\end{lem}
\begin{proof}
Observe first that 
\begin{equation}
\label{AB:BA}
\ran(\Id-T_{(B,A)})\simeq (-D)\cap R
\end{equation}
by \eqref{def:D:R}.
\ref{new:f:v:2}:
Since $D=-D$, 
Theorem~\ref{main:conc} and \eqref{AB:BA} yield
 $\ran(\Id-T_{(A,B)})\simeq \ran(\Id-T_{(B,A)})$,
 and the conclusion follows from 
 \eqref{defn:v}.
\ref{new:f:v:3}:
Let $u\in X$. Since $R=-R$, we
obtain the equivalences $u\in D\cap R$
$\iff -u\in -D$ and $-u\in R$
$\iff -u\in (-D)\cap R$
$\iff u\in -((-D)\cap R)$.
Hence, $D\cap R=- ((-D)\cap R)$.
Consequently, $\overline{D\cap R}=- \overline{((-D)\cap R)}$
and $\ri(D\cap R)=- \ri{((-D)\cap R)}$.
Applying Theorem \ref{main:conc},
in view of \eqref{AB:BA}, 
to the pair $(A,B) $ and the pair $(B,A)$,
we conclude
$\ran(\Id-T_{(A,B)})\simeq -\ran(\Id-T_{(B,A)})$.
Thus $\overline{\ran}(\Id-T_{(A,B)})= -\overline{\ran}(\Id-T_{(B,A)})$,
 and the result follows from \eqref{defn:v}.
\ref{new:f:v:4}: 
The hypothesis implies $D=X=\overline{X}=\overline{D}$.
Now combine with \ref{new:f:v:2} and Proposition~\ref{p:kaos}\ref{main:c:1}.
 \ref{new:f:v:5}: 
That either $\dom A$ or $\dom B$ is bounded, 
implies that $\ran A=X$ (respectively $\ran B=X$)
(see, e.g., \cite[Corollary~21.21]{BC2011}).
Hence
$R=X=\overline{X}=\overline{R}$.
Now combine with \ref{new:f:v:3} and
Proposition~\ref{p:kaos}\ref{main:c:1}. 
\end{proof}

\begin{example}
Suppose that $X=\RR$. 
It follows from Fact~\ref{eq:norm} that
 $v_{(A,B)}=\pm v_{(B,A)}$.
\end{example}

In \cite[Section~3]{Sicon2014}, we constructed examples where
\begin{equation}
\frac{\innp{v_{(A,B)},v_{(B,A)}}}{\norm{v_{(A,B)}}
\norm{v_{(B,A)}}}\in \stb{-1,0,1}.
\end{equation}
We now show that this quotient can take on any value in 
$\left[-1,1 \right]$.

\begin{example}[{angle between $v_{(A,B)}$ and $v_{(B,A)}$}]
Suppose that $S$ is a linear subspace 
of $X$ such that
$\stb{0}\subsetneqq S\subsetneqq X$. 
Let $\theta \in \RR$, let $u\in S$,
 and let $v\in S^\perp$ such that $\norm{u}=\norm{v}=1$.
Set
$a=\sin (\theta) v$,
 and set
$b=\cos (\theta) u$.
Suppose that $A=N_{S+a}$ and that $B=N_S+b$.
Then $D=\dom A-\dom B=S+a-S=S+a$,
and $R=\ran A+\ran B=S^{\perp}+S^{\perp}+b=S^{\perp}+b$.
Consequently, $-D=S-a$.
Clearly, $\overline{D\cap R}=D\cap R=\stb{b+a}$, 
whereas $\overline{(-D)\cap R}=(-D)\cap R
=\stb{b-a}$.
Therefore, $v_{(A,B)}=b+a$,
and $v_{(B,A)}=b-a$.
By Fact~\ref{eq:norm}  $\norm{v_{(A,B)}}=\norm{v_{(B,A)}}=1$.
Moreover, since $a\perp b$
\begin{align}
\scal{v_{(A,B)}}{v_{(B,A)}}=
\innp{b+a,b-a}=\norm{b}^2-\norm{a}^2
=\cos^2 (\theta)-\sin^2 (\theta)=\cos(2\theta).
\end{align}
\end{example}

\subsection{Subdifferential operators}

We now turn to subdifferential operators.

\begin{cor}
Let $f:X\to \RRR$ and $g:X\to \RRR$ be
proper lower semicontinuous convex
functions. Then the following hold:
\begin{enumerate}
\item\label{sub:diff:case}
$
\ran (\Id-T_{(\pt f,\pt g)})\simeq 
(\dom f-\dom g)\cap(\dom f^*+\dom g^*).
$
\item \label{sub:diff:T}
$
\ran T_{(\pt f,\pt g)}\simeq 
(\dom f-\dom g^*)\cap(\dom f^*+\dom g).
$
\end{enumerate}
\end{cor}
\begin{proof}
It is well-known that (see, e.g., \cite[Corollary~16.29]{BC2011})
$\overline{\dom}~f=\overline{\dom} ~\pt f$.
Since $f$ is convex, so is $\dom f$. 
Moreover, by Fact~\ref{D-R-nc}
$\dom \pt f$ is nearly convex.
Therefore, applying Fact~\ref{ri:part}  
to the sets ${\dom f}$ and ${\dom \pt f}$
we conclude that $\dom \pt f\simeq \dom f$.
Using \cite[Proposition~16.24]{BC2011}, and the previous 
conclusion applied to $f^*$, we have
$
\ran \pt f=\dom (\pt f)^{-1}=\dom \pt f^*\simeq\dom f^*
$. Altogether,
\begin{equation}\label{f:p:c:l}
\dom \pt f\simeq \dom f,\quad 
\ran \pt f\simeq \dom f^*.
\end{equation}
Applying Fact~\ref{prox:part} with $C_1=\dom f$,
$C_2=-\dom g$, $D_1=\dom\pt f$,
$D_2=-\dom \pt g$, we conclude that
\begin{equation}\label{D1}
\dom \pt f-\dom \pt g \simeq \dom f-\dom g.
\end{equation}
One shows similarly that
\begin{equation}\label{R1}
\ran \pt f+\ran \pt g\simeq \dom f^*+\dom g^*.
\end{equation}
It follows from 
 the maximal monotonicity of 
$\pt f$ and $\pt g $ and 
Lemma~\ref{lem:HB}\ref{lem:HB:1} that
$(\dom \pt f-\dom \pt g)
\cap (\ran \pt f+\ran \pt g)\neq\fady$.
Applying Corollary~\ref{eq:int} with 
$C_1:=\dom \pt f-\dom \pt g$,
$C_2:= \ran \pt f+\ran \pt g$,
$D_1:=\dom  f-\dom  g$,
and $D_2:= \dom f^*+\dom g^*$,
we conclude that
\begin{equation}\label{prox:case}
(\dom \pt f-\dom \pt g)
\cap (\ran \pt f+\ran \pt g)\simeq (\dom  f-\dom  g)
\cap (\dom f^*+\dom g^*).
\end{equation}
To complete the proof, notice that
by Fact~\ref{subd:star}
$\pt f$ and $\pt g$ are $3^*$ monotone 
operators, and by assumption 
$\pt f+\pt g$ is maximally monotone.
Therefore, by Theorem~\ref{main:conc},
we have 
\begin{equation}\label{prox:case:2}
\ran (\Id-T_{(\pt f,\pt g)})\simeq(\dom \pt f-\dom \pt g)
\cap (\ran \pt f+\ran \pt g).
\end{equation}
Combining \eqref{prox:case} and \eqref{prox:case:2}
we conclude that \ref{sub:diff:case}
holds true.
To prove \ref{sub:diff:T},
combine Corollary~\ref{ran:T}, \eqref{f:p:c:l}
and Corollary~\ref{eq:int}.
\end{proof}
\begin{cor}
Let $f:X\to \RRR$ be 
proper, convex, lower semicontinuous
and suppose that $V$ is a nonempty 
closed convex subset of $X$.
Suppose that $A=\pt f$ and
$B=N_V$. Then the following hold:
\begin{enumerate}
\item\label{prox:n:1}
$
T_{(A,B)}=J_{N_V}R_{\pt f}+\Id-J_{\pt f}
=\PR_V(2\prox_f-\Id)+\Id-\prox_f.
$
\item\label{prox:n:2}
$
\ran (\Id-T_{(A,B)})\simeq (\dom f-V)\cap (\dom f^*+(\rec V)^\ominus).
$
\item\label{prox:n:3}
$
\ran T_{(A,B)}\simeq (\dom f-(\rec V)^\ominus)\cap (\dom f^*+V).
$
\end{enumerate}
Consequently, if $V$ is a linear subspace we may add to this list
the following items:
\begin{enumerate}
\setcounter{enumi}{3}
\item\label{prox:n:4}
$
\ran (\Id-T_{(A,B)})\simeq (\dom f+V)\cap (\dom f^*+V^\perp).
$
\item\label{prox:n:5}
$
\ran T_{(A,B)}\simeq (\dom f+V^\perp)\cap (\dom f^*+V).
$
\end{enumerate}
\end{cor}
\begin{proof}
Since $\ran N_V$ is nearly convex
and $(\rec V)^{\ominus}$ is convex, 
it follows from \eqref{inv:r:d}, Fact~\ref{rec:PC} 
and Fact~\ref{ri:part} that
\begin{equation}\label{rec:eq}
\ran N_V\simeq (\rec V)^\ominus.
\end{equation}
\ref{prox:n:1}: This follows from \eqref{eq:dual}
and the fact that $J_{N_V}=\PR_V$
and $J_{\pt f}=\prox f$. 
\ref{prox:n:2}: Combine \eqref{f:p:c:l}, \eqref{rec:eq}, 
Theorem~\ref{main:conc}, Fact~\ref{prox:part}
 and Corollary~\ref{eq:int}.
\ref{prox:n:3}: Combine \eqref{f:p:c:l}, \eqref{rec:eq}, 
Corollary~\ref{ran:T}, Fact~\ref{prox:part}
 and Corollary~\ref{eq:int}.
 \ref{prox:n:4} and \ref{prox:n:5}:
 It follows from \cite[Proposition~6.22 and Corollary~6.49]{BC2011}
 that $\rec V=V$ and $(\rec V)^\ominus=V^\perp$.
Combining this with \ref{prox:n:2} and \ref{prox:n:3}, we 
obtain \ref{prox:n:4} and \ref{prox:n:5}, respectively.
\end{proof}

\begin{cor}[two normal cone operators] 
\label{n:cone}
Let $U$ and $V$
be two nonempty closed convex subsets of $X$, and
suppose that $A=N_U$ and that $B=N_V$.
Then the following hold:
\begin{enumerate}
\item\label{n:cone:2}
$\ran (\Id-T_{(A,B)})\simeq (U-V)\cap
((\rec U)^{\ominus}+(\rec V)^{\ominus})$.
\item\label{n:cone:1}
$\ran T_ {(A,B)})\simeq(U-(\rec V)^\ominus)
\cap((\rec U)^{\ominus}+V)$.
\item\label{n:cone:3}
$v_{(A,B)}=\PR_{\overline{U-V}}(0).$
\end{enumerate}
\end{cor}
\begin{proof}
Clearly $\dom A=U$ and $\dom B=V$.
It follows from \eqref{rec:eq} 
that $\ran N_U\simeq (\rec U)^\ominus$
and $\ran N_V\simeq (\rec V)^\ominus$.
Therefore, Fact~\ref{prox:part} implies that 
\begin{equation}\label{R:n:c}
R \simeq (\rec U)^\ominus+(\rec V)^\ominus.
\end{equation}
Now \ref{n:cone:2} follows from combining 
\eqref{R:n:c} and Theorem~\ref{main:conc},
and \ref{n:cone:1} follows from combining
\eqref{rec:eq} applied to the sets $U$ and $V$,
 Fact~\ref{prox:part}
 and Corollary~\ref{ran:T}.
It remains to show \ref{n:cone:3} is true.
Set $v=\PR_{\overline{U-V}}(0)=\PR_{\overline{D}}(0)$.
On the one hand, by definition of $v$ 
and Proposition~\ref{p:kaos}\ref{main:c:1}, we have 
$v_{(A,B)}\in \overline{D}\cap \overline{R}\subseteq \overline{D}$ 
and hence
\begin{equation}\label{v:v}
\norm{v}\le \norm{v_{(A,B)}}.
\end{equation}
On the other hand,
using {\rm\cite[Corollary~2.7]{BCL04}} we have
$v\in \overline{(\PR_U-\Id)(V)}\cap\overline{(\Id-\PR_V)(U)}
\subseteq (\rec U)^\oplus\cap(\rec V)^\ominus$.
Therefore, using
 \eqref{rec:eq} and that $0\in (\rec U)^\ominus$ we have
$v\in  (\rec U)^\oplus\cap(\rec V)^\ominus
\subseteq (\rec V)^\ominus \subseteq (\rec U)
^\ominus+(\rec V)^\ominus\subseteq\overline{R}.
$
Hence, 
\begin{equation}\label{v:v:v}
v\in \overline{D}\cap \overline{R}.
\end{equation}
Combining \eqref{v:v}, \eqref{v:v:v}
and  Proposition~\ref{p:kaos}\ref{main:c:1}
yields $v=v_{(A,B)}$.
\end{proof}

\subsection{Firmly nonexpansive mappings}

We now restate the main result from the perspective of fixed point
theory.

\begin{cor}\label{fix:p:t}
Let $T_1:X\to X$ and $T_2:X\to X$
be firmly nonexpansive such that  each $T_i$ satisfies
\begin{equation}\label{T:*}
(\forall x\in X)(\forall y\in X) 
\qquad \inf_{z\in X} \innp{T_ix-T_iz,(y-T_iy)-(z-T_iz)}>+\infty,
\end{equation}
and set $T:=T_2(2T_1-\Id)+\Id-T_1$.
Then 
\begin{equation}\label{T1:T2:1}
\ran T\simeq(\ran T_1-\ran(\Id-T_2))\cap(\ran(\Id-T_1)+\ran T_2),
\end{equation}
and
\begin{equation}\label{T1:T2:2}
\ran (\Id-T)\simeq (\ran T_1-\ran T_2)
\cap(\ran(\Id-T_1)+\ran(\Id-T_2)).
\end{equation}
\end{cor}
\begin{proof}
Using Fact~\ref{T:JA} we conclude that 
there exist maximally monotone operators 
$A:X\rras X$ and $B:X\rras X$
such that 
\begin{equation} 
T_1=\bJ_A \quad \text{and }\quad T_2=\bJ_B.
\end{equation}
Moreover, it follows from {\rm\cite[Theorem~2.1(xvii)]{BMW2}}
 and \eqref{T:*} that
that $A$ and $B$ are $3^*$ monotone.
By \eqref{e:done},  we conclude that 
$T=T_{(A,B)}$. Using Corollary~\ref{ran:T},
Fact~\ref{F:res:ran} and \eqref{inv:r:d}
we have
\begin{align*}
\ran T&\simeq (\dom A-\ran B)\cap(\ran A+\dom B)\\
&=(\ran(\Id-T_1)-\ran T_2)\cap(\ran T_1+\ran(\Id-T_2)).
\end{align*}
That is, \eqref{T1:T2:1} holds true. 
Similarly, one can prove \eqref{T1:T2:2}
by combining Theorem~\ref{main:conc},
Fact~\ref{F:res:ran} and \eqref{inv:r:d}.
\end{proof}

\section{Some infinite-dimensional observations}

\label{S:infd}

In this final section, we provide some results
that remain true in infinite-dimensional settings.
We assume henceforth that 
\begin{empheq}[box=\mybluebox]{equation}
\text{$H$ is a (possibly infinite-dimensional) real Hilbert
space.}
\end{empheq}
A pleasing identity arises when the we are dealing
with normal cone operators of closed subspaces.

\begin{prop}\label{Con:subsp}
Let $U$ and $V$ be closed linear subspaces 
of $H$, and suppose that $A=N_U$ and $B=N_V$. Then  
$\ran({\Id -T_{(A,B)}})=(U+V)\cap ( U^\perp +V^\perp)$.
\end{prop} 
\begin{proof}
Since $\gr N_U=U\times U^\perp$ 
and $\gr N_V=V\times V^\perp$,
the result 
follows from \eqref{con:eq}.
\end{proof}

\begin{prop}\label{cn:ex3}
Let $\U$ and $\E$ be closed linear
subspaces of $H$ such that 
\begin{equation}\label{ee1}
\U^{\perp} \cap \E=\stb{0},
\end{equation}
and suppose that $A=N_{\U}$ and $B=\PR_\E$. 
Then the following hold:
\begin{enumerate}
\item\label{n:ap14:1}
$
\U^\perp\cap\PR^{-1}_\E(\PR_\E(\U^\perp)\setminus \PR_\E(\U))
\subseteq (D\cap R)\setminus \ran (\Id-T_{(A,B)}).
$
\item
\label{n:ap14:2}
$
(\U^{\perp}+\E)\cap(\E+\PR^{-1}_\E(\PR_\E(\U^\perp)
\setminus \PR_\E(\U)))
\subseteq (\ran A+\ran B)\setminus \ran (A+B).
$
\end{enumerate}
Consequently, if 
$\PR_\E(\U^\perp)\setminus \PR_\E(\U)\neq \fady$,
then
\begin{equation}\label{int:con:ex}
\ran(\Id-T_{(A,B)})\subsetneqq D\cap R
\end{equation}
and
\begin{equation}\label{BH:inf:ex}
 \ran (A+B)\subsetneqq \ran A+\ran B.
\end{equation}
\end{prop}
\begin{proof}
It is clear that $D\cap R
=(\U-H)\cap (\U^{\perp}+\E)=\U^{\perp}+\E$. 
Notice that \ref{n:ap14:1} and \ref{n:ap14:2}
trivially hold when 
$\PR_\E(\U^\perp)\setminus \PR_\E(\U)=\fady$.
Now suppose that 
$\PR_\E(\U^\perp)\setminus \PR_\E(\U)\neq\fady$.
It follows from \eqref{defn:Zw} and \eqref{ee1} that
 $(\forall w\in \U^\perp\subseteq \U^{\perp}+\E)$
\begin{align}\label{con:p}
Z_w\neq \fady
&\iff(\exists \uu\in \U) \text{~such that~} 
w\in N_\U\uu+\PR_\E (\uu-w)\nonumber\\
&\iff (\exists \uu\in \U) \PR_\E \uu-\PR_\E w\in \U^\perp-w=\U^\perp\nonumber\\
&\iff(\exists \uu\in \U)\PR_\E \uu-\PR_\E w=0\iff \PR_\E w\in \PR_\E(\U).
\end{align}
Now let $w\in \U^\perp\subseteq \U^{\perp}+\E=D\cap R$
such that $ \PR_\E w\not \in \PR_\E(\U)$.
Then \eqref{con:p} implies that
$Z_w=\fady$, hence by 
\eqref{def:zw} $w\not \in \ran (\Id-T_{(A,B)})$,
which proves \ref{n:ap14:1}
and consequently \eqref{int:con:ex}.
To complete the proof we need to show that
\ref{n:ap14:2} holds.
Notice that
 $(\forall \uu^\perp\in \U^\perp)$ 
 $\uu^\perp+\PR_\E\uu^\perp\in
\U^\perp+\E=\ran A+\ran B$.
 It follows from \eqref{ee1} that
\begin{align}\label{con:p:2}
\uu^\perp+\PR_\E\uu^\perp\in \ran (A+B)
&\iff (\exists \uu\in \U= \dom (A+B))~ 
\uu^\perp+\PR_\E\uu^\perp\in 
\U^\perp+\PR_\E\uu\nonumber\\
&\iff(\exists \uu\in \U) \PR_\E\uu^\perp-\PR_\E\uu\in \U^\perp-\uu^\perp=\U^\perp\nonumber\\
&\iff (\exists \uu\in \U)\PR_\E\uu^\perp=\PR_\E\uu.
\end{align}
Now, let 
$\uu^\perp \in\U^\perp$
such that $ \PR_\E \uu^\perp\not \in \PR_\E(\U)$.
Then using \eqref{con:p:2}
$w =\uu^\perp+\PR_\E\uu^\perp\not\in \ran (A+B)$.
Notice that by construction 
$w\in \U^{\perp}+\E=\ran A+\ran B$.
\end{proof}

\begin{rem}
Notice that in Proposition~\ref{cn:ex3}
both $A$ and $B$ are linear relations,
maximally monotone and
$3^*$ monotone operators. 
Consequently, the sets $D$ and $R$
 are linear subspaces of $H$. 
 When $H$ is finite-dimensional,
Corollary~\ref{aff:case} and {\rm\cite[footnote on page~174]{Br-H} }
imply that 
 $\ran (\Id-T)=D\cap R$ and $\ran (A+B)=\ran A+\ran B$. 
 Thus, if 
 \eqref{int:con:ex} or \eqref{BH:inf:ex}
 holds, then $H$ is necessarily infinite-dimensional.
\end{rem}

We now
provide a concrete example 
 in $\ell^2(\NN)$ where 
both \eqref{int:con:ex} and 
\eqref{BH:inf:ex} hold.
This illustrates again the requirement 
of the closure in Fact~\ref{Br-H}.

\begin{prop}\label{cn:ex5}
Suppose that $H=\ell^2(\NN)$, 
let $p\in \RR_{++}$, and let $(\alpha_n)_{n\in \NN}$
be a sequence in $\RR_{++}$ such that
\begin{equation}\label{al:asm}
\alpha_n\to 0,
\quad 
\sum_{n=0}^\infty \alpha^{2p-2}_n<+\infty
\quad 
\text{and}
\quad
\sum_{n=0}^\infty \alpha^{2p-4}_n=+\infty.
\end{equation}
Set 
$
\U=\menge{x=(x_n)_{n\in \NN}\in H}{x_{2n+1}=-\alpha_n x_
{2n}}
$
 and $
\E=\menge{x=(x_n)_{n\in \NN}\in H}{x_{2n}=0}
$, and
suppose that $A=N_{\U}$ and $ B=\PR_\E$. 
Then $\PR_\E(\U^\perp)\setminus \PR_\E(\U)\neq \fady$
and hence 
\begin{equation}\label{int:con:ex:d}
\ran(\Id-T_{(A,B)})\subsetneqq D\cap R
\quad\text{and}\quad
 \ran (A+B)\subsetneqq \ran A+\ran B.
\end{equation}
\end{prop}
\begin{proof}
It is easy to check that
$\U^\perp=\menge{x=(x_n)\in H}{x_{2n+1}
=\alpha_n^{-1} x_{2n}}$. Hence
$\U^\perp\cap \E=\stb{0}$.
Let $w\in H$ be defined as
$(\forall n\in \NN)~w_{2n}=\alpha_n^p$
and  $w_{2n+1}=\alpha_n^{p-1}$. Clearly $w\in \U^\perp$.
We claim that $\PR_\E w\not \in \PR_\E \U$.
Suppose this is not true. Then 
$(\exists \uu \in \U)$ such that 
$\PR_\E w= \PR_\E \uu$.
Hence $(\forall n\in \NN)~\uu_{2n+1}= ( \PR_\E \uu)_{2n+1}
=(\PR_\E w)_{2n+1}= w_{2n+1}=
 \alpha_n^{p-1}$. 
Consequently,
$(\forall n\in \NN)~\uu_{2n}=-\alpha_n^{p-2}$, which is absurd
since it implies that 
$\sum_{n=0}^\infty\uu^2_{2n}
=\sum_{n=0}^\infty \alpha^{2p-4}_n=+\infty$,
by \eqref{al:asm}. Therefore,
 $\PR_\E(\U^\perp)\setminus \PR_\E(\U)\neq \fady$.
 Using Proposition~\ref{cn:ex3}
we conclude that 
\eqref{int:con:ex:d} holds.
\end{proof}

The next example is a special case 
of Proposition~\ref{cn:ex5}.

\begin{ex}\label{cn:ex4}
Suppose that $H=\ell^2(\NN)$,
let $(\alpha_n)_{n\in \NN}=(1/(n+1))_{n\in \NN}$,
let $p\in \left]\tfrac{3}{2},\tfrac{5}{2}\right]$ 
and let  $\U$, $\E$,
$A$ and $B$
be as defined in Proposition~\ref{cn:ex5}. 
Since $2p-2>1$ and $2p-4\le 1$,
 we see that \eqref{al:asm} holds.
From Proposition~\ref{cn:ex5} 
we conclude that 
$\ran (\Id-T_{(A,B)})\subsetneqq D\cap R$
and 
$\ran(A+B)\subsetneqq \ran A+\ran B$.
\end{ex}

When $A$ or $B$ has additional structure, it
may be possible to traverse between $\ran(A+B)$
 and $\ran(\Id-T_{(A,B)})$ as we illustrate now.

\begin{prop}\label{p:las:AB}
Let $A:H\rras H$ and $B:H\rras H$
be maximally monotone.
Then the following hold:
\begin{enumerate}
\item\label{it:las:B}
If $B:H\to H$ is linear, then
$\ran(\Id-T_{(A,B)})=J_B(\ran(A+B))$
and
$\ran(A+B)=(\Id+B)\ran(\Id-T_{(A,B)})$.
\item\label{it:las:A}
If $A:H\to H$ is linear and $Id-A$ is invertible, then
$\ran(\Id-T_{(A,B)})=(\Id-A)^{-1}(\ran(A+B))$
and $\ran(A+B)=(\Id-A)\ran(\Id-T_{(A,B)})$.
\item\label{it:las:AB}
If $A:H\to H$ and $B:H\to H$ 
are linear and $A^*=-A$, then
$(\forall \lam \in [0,1])$
$\ran(\Id-T_{(A,B)})=J_{\lam A^*+(1-\lam)B}(\ran(A+B))$. 
\end{enumerate}
\end{prop}
\begin{proof}
Let $w\in X$. It follows from \eqref{def:zw}
that 
\begin{equation}\label{e:uni}
w\in \ran (\Id-T_{(A,B)})\iff (\exists x\in H) 
\quad \text{such that}\quad
w\in Ax+B(x-w).
\end{equation}
\ref{it:las:B}:
It follows from \eqref{e:uni} that 
$w\in \ran (\Id-T_{(A,B)})\iff (\exists x\in H) $
such that $ w\in Ax+Bx-Bw\iff  (\exists x\in H)~ (\Id+B)w=w+Bw\in (A+B)x$
$\iff   (\exists x\in H)~ w\in J_B((A+B)x)$
$\iff w\in J_B(\ran (A+B))$.
Using \cite[Theorem~2.1(ii)\&(iv)]{BMW2}
we learn that $J_B$ is a bijection, hence invertible, 
and $\ran(A+B)=J_B^{-1}\ran(\Id-T_{(A,B)})
=(\Id+B)\ran(\Id-T_{(A,B)})$, as claimed.
\ref{it:las:A}:
It follows from \eqref{e:uni} that 
$w\in \ran (\Id-T_{(A,B)})\iff (\exists x\in H) $
such that
$w-Aw\in A(x-w)+B(x-w)=(A+B)(x-w)$
$\iff  (\exists x\in H)~(\Id-A)w\in (A+B)(x-w)$
$\iff  (\exists x\in H)~w\in (\Id-A)^{-1}((A+B)(x-w))$
$\iff w\in (\Id-A)^{-1}(\ran (A+B))$.
Since
$\Id-A$ is invertible, we learn that $\Id-A$ is a bijection
and $\ran(A+B)=(\Id-A)\ran(\Id-T_{(A,B)})$, as claimed.
\ref{it:las:AB}:
It follows from \eqref{e:uni}
that $w\in \ran (\Id-T_{(A,B)})\iff (\exists x\in H) $ such that
$(\forall \lam \in [0,1])$
$w-\lam Aw+(1-\lam)Bw
=A(x-\lam w)+B(x-w+(1-\lam )w)\in (A+B)(x-\lam w)$
$\iff (\Id+\lam A^*+(1-\lam )B)w\in \ran(A+B)$
$\iff w\in J_{\lam A^*+(1-\lam)B}(\ran(A+B))$.
\end{proof}

\section*{Acknowledgments}
HHB was partially supported by the Natural Sciences and Engineering
Research Council of Canada and by the Canada Research Chair Program.
WLH was partially supported by the Natural Sciences and Engineering
Research Council of Canada. WMM
was partially supported by the Natural Sciences and Engineering
Research Council of Canada of HHB and WLH.

\end{document}